\documentclass[12pt,reqno,a4paper]{amsart}

\usepackage{amsmath,amsfonts,amsthm,amssymb,amsxtra,enumerate,mathtools,mathabx}
\usepackage[normalem]{ulem}
\usepackage[utf8]{inputenc}
\usepackage[T1]{fontenc}
\usepackage{lmodern}
\usepackage{bbm}
\usepackage{todonotes}
\usepackage{mathrsfs}
\usepackage[many]{tcolorbox}  
\usepackage{enumerate}
\usepackage{csquotes}
\usepackage{cite}
\usepackage{mathdots}
\RequirePackage{doi}
\usepackage{hyperref}

\newtheorem{theorem}{Theorem}
\newtheorem{proposition}[theorem]{Proposition}
\newtheorem{lemma}[theorem]{Lemma}

\theoremstyle{definition}

\theoremstyle{remark}

\newtheorem{remark}[theorem]{Remark}

\numberwithin{equation}{section}

\newcommand{\dd}{\, \mathrm{d}}

\renewcommand{\epsilon}{\varepsilon}

\newcommand{\N}{\mathbb{N}}

\renewcommand{\phi}{\varphi}
\newcommand{\R}{\mathbb{R}}

\DeclareMathOperator{\Tr}{Tr}

\DeclareMathOperator{\sech}{sech}

\begin{document}
	
	\title[Lieb-Thirring type bounds for perturbed Schr\"odinger operators 
	]{Lieb-Thirring type bounds for perturbed Schr\"odinger operators with single-well potentials}
	
	\author{Larry Read} 
	\address{Larry Read, Department of Mathematics, Imperial College London, London SW7 2AZ, UK}
	\email{l.read19@imperial.ac.uk}
	
	
	\begin{abstract}
		We prove an upper bound on the sum of the distances between the eigenvalues of a perturbed Schr\"odinger operator $H_0-V$ and the lowest eigenvalue of $H_0$. Our results hold for operators $H_0=-\Delta-V_0$ in one dimension with single-well potentials. We rely on a variation of the well-known commutation method. In the P\"oschl-Teller and Coulomb cases we are able to use the explicit factorisations to establish improved bounds.
	\end{abstract}

	\maketitle

	\section{Introduction}
		Consider the self-adjoint Schr\"odinger operator $H_V= H_0-V$, with $H_0=-\Delta$, on $L^2(\R^d)$. When $V$ has suitable decay then $H_V$ has a negative spectrum which is discrete and consists of finite or infinitely many eigenvalues $\{E_k(H_V)\}_{k=1}$. The celebrated Lieb-Thirring inequality \cite{lieb1976inequalities} states that 
		\begin{equation} \label{eqn:liebthir}
			\sum_{k} |E_k(H_V)|^\gamma\leq L_{\gamma,d}\int_{\R^d}V(x)_{+}^{\gamma+\frac{d}{2}}\dd x
		\end{equation}
		with $\gamma\geq 1/2$ if $d=1$, $\gamma>0$ if $d=2$ and $\gamma\geq 0$ for $d\geq 3$, where $V_{+}=(|V|+V)/2$. The endpoint cases for $d=1$ and $d\geq 3$ were settled in \cite{weidl1996lieb} and \cite{cwikel1977weak,lieb1976bounds,rozenbljum1972distribution} respectively.
		
		Substantial effort has been made to determine the sharp constants in (\ref{eqn:liebthir}) (refer to \cite{schimmer2022state} or \cite{frank2020lieb} for a review). Considerable progress was made by A.Laptev and T.Weidl in \cite{laptev2000sharp} in which they established that $L_{\gamma,d}=L_{\gamma,d}^{\mathrm{cl}}$ for all $\gamma\geq 3/2$ and $d\geq 1$, where
		\begin{equation*}
		    L_{\gamma,d}^{\mathrm{cl}}=\frac{\Gamma(\gamma+1)}{(4\pi)^{d/2}\Gamma(\gamma+1+\frac{d}{2})}
		\end{equation*}
		is the semi-classical constant. Their method relied on a proof of this fact with $\gamma=3/2,\ d=1$ for operator-valued potentials, which they found by generalising a trace formula of V.Buslaev, L.Faddeev and V.Zaharov. An application of this, using separation of variables, settled the scalar result for $d\geq 1$, after which it followed for all $\gamma\geq 3/2$ by a standard lifting argument of M.Aizenman and E.Lieb \cite{aizenman1978semi}.
		
		Shortly after, a simple proof of their result was found by R.Benguria and M.Loss in \cite{benguria1999simple} where the authors used the commutation method to establish the sharp bound for $\gamma=3/2,\ d=1$. This scheme, which we detail below, involves factorising Schr\"odinger operators by so-called \enquote*{creation and annihilation operators.} It has been known to authors as far back as C.Jacobi \cite{Jacobi1837}. Though its first known application to trace inequalities was by U-W.Schmincke in \cite{schmincke1978schrodinger} where the author derived the lower bound
		\begin{equation*}
			\sum_k \sqrt{|E_k(H_V)|} \geq \frac{1}{4}\int_{\R^d} V(x) \dd x.
		\end{equation*}
		Since its elegant use by R.Benguria and M.Loss to obtain sharp Lieb-Thirring inequalities it has been employed to the Robin boundary case on the half-line by P.Exner, A.Laptev and M.Usman in \cite{exner2014some} to obtain an analogous result. A bound which was improved by L.Schimmer in \cite{schimmer2019improved} using a variation of this idea, known as the double commutation method.
		
		Recently, A.Laptev in \cite{laptev2021factorisation} found that the factorisation scheme could be used to derive the Hardy-Lieb-Thirring bound on the half-line, with a conjectured sharp constant. This inequality is the improvement of \eqref{eqn:liebthir} found by T.Ekholm and R.Frank in \cite{ekholm2008lieb} which, for the Schr\"odinger operator in $L^2(\R_+)$ with Dirichlet boundary conditions, states that
	    \begin{equation*}\label{eqn:Hardyliebthirring}
	        \Tr\left(-\frac{\mathrm{d}^2}{\mathrm{d}r^2}-\frac{1}{4r^2}-V(r)\right)_{-}^\gamma\leq L_{\gamma,d}^{\mathrm{HLT}}\int_{\R_+}V(r)_{+}^{\gamma+\frac{1}{2}}\dd r
	    \end{equation*}
	    for $\gamma\geq 1/2$. Where an analogous result holds on the whole space for $d\geq 3$ \cite{ekholm2006lieb}. In this paper we aim to generalise this application of the commutation method. 
		
	    It is natural to consider whether other improvements of \eqref{eqn:liebthir} exist for perturbed operators $H_V=H_0-V$ where $H_0=-\Delta-V_0$. An extensive solution was given by R.Frank, B.Simon and T.Weidl in \cite{frank2007eigenvalue}. They found that when $H_0$ is such that $\inf \mathrm{spec}(H_0)=0$ and has a \enquote*{regular ground state} then there exists $\beta=\beta(V_0)>0$ for which
		\begin{equation}\label{eqn:franksimonweidl}
			\sum_k |E_k(H_V)|^\gamma\leq \beta(V_0)^{\gamma+\frac{d}{2}}L_{\gamma,d} \int_{\R^d}V(x)_{+}^{\gamma+\frac{d}{2}}\dd x   
		\end{equation}
		with $\gamma,d$ and $L_{\gamma,d}$ as in \eqref{eqn:liebthir}. In fact, their result is a much stronger comparison of individual eigenvalues. The regular ground state condition requires the existence of a strictly positive bounded solution $u_0$ to
		\begin{equation*}
		    (-\Delta-V_0)u_0=0,
		\end{equation*}
		in which case $\beta$ in \eqref{eqn:franksimonweidl} is given as $\beta(V_0)=(\sup u_0(x)/\inf u_0(x))^2$. As noted in \cite{frank2007eigenvalue} this property is satisfied for a large class of $V_0$, including periodic potentials. 
		
		In dimensions one and two improvements of the sort \eqref{eqn:franksimonweidl}, which are homogeneous in $V$, are impossible. For any positive and non-trivial $V_0$, $H_0$ always generates a negative eigenvalue. Nonetheless, one can shift the operator $H_0$ by the ground energy to obtain $\inf\mathrm{spec}(H_0+|E_1(H_0)|)=0$ and \eqref{eqn:franksimonweidl} becomes
		\begin{equation}\label{eqn:FSWformal}
		    \sum_k\left(|E_k(H_V)|-|E_1(H_0)|\right)_+^\gamma\leq\beta^{\gamma+d/2}L_{\gamma,d}\int_{\R^d}V(x)_+^{\gamma+d/2}\dd x.
		\end{equation}
		However, since $\inf\mathrm{spec}(H_0+|E_1(H_0)|)$ belongs the discrete spectrum, $H_0+|E_1(H_0)|$ does not have a regular ground state with finite $\beta(V_0)$ (see \cite[Theorem C.4.2]{Simon1982}). In what follows, we will generalise the method of A.Laptev in \cite{laptev2021factorisation} to show that a bound of this type exists for a class of potentials with a finite sharp constant.  
		\section{Main results}
		Our first result is an upper bound on the sum of the distance between the eigenvalues of $H_V$ and the first eigenvalue of the Schr\"odinger operator $H_0$, where $V_0$ is a single-well potential. By single-well we mean a function that is non-decreasing up to some point, and non-increasing thereafter. 
		\begin{theorem}\label{thm:firsteigen}
			Let $V_0$ be positive and single-well with $(1+|x|)V_0\in L^1(\R)$ and let $V_{+}\in L^2(\R)$. Then for any $\gamma\geq 3/2$ the following holds
			\begin{equation}  \label{eqn:thm1state} 
			  \sum_k\left(\sqrt{|E_k(H_V)|}-\sqrt{|E_1(H_0)|}\right)_{+}^{2\gamma}\leq L_\gamma \int_{\R} V(x)_{+}^{\gamma+1/2} \dd x.
			\end{equation}
			where $L_{\gamma}$ is independent of $V_0$ and $V$, with $L_{3/2}=L_{3/2,1}^{cl}$.
		\end{theorem}
		For $\gamma=3/2$ the inequality is optimal in the sense that as $V_0$ is taken to zero what remains is precisely the sharp Lieb-Thirring inequality. Though the bound differs to that deduced formally from \cite{frank2007eigenvalue}, \eqref{eqn:FSWformal}, we have explicit constants which are independent of $V_0$. 
		
		As a consequence of our proof we find the same result on the half-line for the Schr\"odinger operators $H_{0,\sigma}=-\frac{\mathrm{d}^2}{\mathrm{d}r^2}-V_0$ and $H_{V,\sigma}=H_0-V$ on $L^2(\R_{+})$ with Robin boundary conditions $\phi^\prime(0)-\sigma\phi(0)=0$, $\sigma\in \R$.
		\begin{theorem}\label{thm:Robincase}
		  Let $V_0$ be positive and non-increasing with $V_0\in L^1(\mathbb{R}_+)$ and let $V_{+}\in L^2(\R_{+})$. Then for any $\gamma\geq 3/2$ the following holds 
		  \begin{equation}\label{eqn:thm2}	\sum_k\left(\sqrt{|E_k(H_{V,\sigma})|}-\sqrt{|E_1(H_{0,\sigma})|}\right)_{+}^{2\gamma}\leq 2L_\gamma \int_{\mathbb{R}_+} V(r)_{+}^{\gamma+1/2} \dd r.
			\end{equation}
			where $L_\gamma$ is the same as in \eqref{eqn:thm1state}.
		\end{theorem}
		The inequality we obtain through our factorisation method has a more general form than theorems 1 and 2. In both cases, the results follow by showing that the ground state of $H_0$ ($H_{0,\sigma}$) is log-concave for single-well (non-increasing) potentials. Where we have an explicit factorisation for $H_0$ we are able to show something stronger and improve (\ref{eqn:thm1state}) and (\ref{eqn:thm2}). The following two results show that this is true for the Coulomb and P\"oschl-Teller potentials. We identify $\sigma=\infty$ for the Robin operators with Dirichlet boundary conditions. 
		\begin{theorem}\label{thm:Coulomb}
		    Let $V_0(x)=\kappa/r-\nu(\nu+1)/r^2$ with $\nu\geq-1/2$, $\kappa>0$ and let $V_+\in L^2(\R_{+})$. Then the following holds
			\begin{equation*}
				\sum_{k=1}\left(\sqrt{|E_k(H_{V,\infty})|}-\sqrt{|E_k(H_{0,\infty})}|\right)^3\leq \frac{3}{8}\int_{\R_{+}} V(r)_+^2 \dd r
			\end{equation*}
			where $E_k(H_{0,\infty})=-\kappa^2/4(\nu+k)^2$ are the eigenvalues of $H_{0,\infty}$. 
		\end{theorem}
		The P\"oschl-Teller case is different in that the potential has finitely many eigenvalues and holds on the whole of $\R$. In this case we obtain a direct improvement of \eqref{eqn:liebthir}.
		\begin{theorem}\label{thm:soliton}
			Let $V_0(x)=\nu(\nu+1)\sech^2(x)$ with $\nu>0$ and let $V_+\in L^2(\R)$. Then the following holds
			\begin{equation*}
                \sum_{k=1}^{\lceil\nu\rceil}\left(\sqrt{|E_k(H_V)|}-\sqrt{|E_k(H_0)|}\right)^3+\sum_{k=\lceil\nu\rceil+1} |E_k(H_V)|^{3/2}\leq \frac{3}{16}\int_{\R} V(x)_+^2 \dd x
			\end{equation*}
			where $E_k(H_0)=-(\nu-k+1)^2$ are the eigenvalues of $H_0$. In particular,
			\begin{equation}\label{eqn:polsch2nd}
			    \sum_{k=\lceil\nu\rceil+1}|E_k(H_V)|^\gamma\leq L_{\gamma,1}^{\mathrm{cl}}\int_{\R} V(x)_+^{\gamma+1/2} \dd x
			\end{equation}
			holds for all $\gamma\geq3/2$.
		\end{theorem}
		The first inequality contains all discrete spectral data from $H_V$, by removing the first $\lceil \nu\rceil$ terms we can apply a Aizenman-Lieb lifting argument to obtain \eqref{eqn:polsch2nd}. The latter is an improvement of the standard Lieb-Thirring inequality, it implies that for any $N\in\N$
		\begin{equation*}
		    \sum_{k=N}\Big\vert E_k\left(-\frac{\mathrm{d}^2}{\mathrm{d}x^2}-V\right)\Big\vert^\gamma\leq L_{\gamma,1}^{\mathrm{cl}}\int_{\R}\left(V(x)-N(N-1)\mathrm{sech}^2(x)\right)_{+}^{\gamma+\frac{1}{2}}\dd x.
		\end{equation*}

		In what proceeds we begin by introducing the factorisation scheme for the perturbed Schr\"odinger operator $H_V$. We prove a technical proposition which will be the general inequality from which we derive theorems 1-4. In Section \ref{section:thm12} we prove that the single-well property for $V_0$ is sufficient to apply our bound and we prove theorems 1 and 2 for the stated restrictions on $V_0$. Finally, in Section \ref{section:polcou} we present the explicit factorisations of the P\"oschl-Teller and Coulomb operators and formulate the improved results in these cases.
	\section{Factorisation scheme for $H_{V}$} \label{Section:facscheme}
		We start with the assumption that $V_0,V\in C_0^\infty(\R)$ and $V_0,V\geq0$. The corresponding operators $H_0=-\frac{\mathrm{d}^2}{\mathrm{d}x^2}-V_0$ and $H_V=H_0-V$ on $L^2(\R)$ have negative eigenvalues which we denote by $\{E_k(H_0)\}_{k=1}^M$ and $\{E_k(H_V)\}_{k=1}^N$ in increasing order (it is clear that $N\geq M$). For the sake of brevity, we let $\mu_k=-E_k(H_0)$ and $\lambda_k=-E_k(H_V)$, unless we wish to distinguish the underlying operator.
		
		The factorisation scheme for $H_0$ goes as follows: let $\phi_1$ be the first eigenfunction of $H_0$ corresponding to $-\mu_1$. It is well-known that $\phi_1$ can be taken to be positive without zeros. Consider the first order operator $D_1=-\frac{\mathrm{d}}{\mathrm{d}x}-g_1$, where $g_1(x)=\phi_1^\prime(x)/\phi_1(x)$. We see that $H_0$ factorises by $D_1$, in the sense that
		\begin{equation}
		\begin{split}\label{EqnD}
		    &D_1^\ast D_1=H_0+\mu_1,\\
		    & D_1 D_1^\ast =H_0-2g_1^\prime+\mu_1.
        \end{split}
		\end{equation}
		This follows by expanding with $D_1^\ast=-\frac{\mathrm{d}}{\mathrm{d}x}-g_1$ and using that $g_1$ solves the Riccati equation
		\begin{equation}\label{EqnRiccatig}
			g_1^\prime 
			=\mu_1 -V_0 -g_1^2
		\end{equation}
		deduced directly from its definition. We note that the relations (\ref{EqnD}) can be seen precisely by looking on the associated quadratic forms, where it holds for all functions in the form core $C_0^\infty(\R)$. 
		
		As a consequence it follows that $D_1^\ast D_1$ has discrete spectrum consisting of the eigenvalues $\{-\mu_k+\mu_1\}_{k=1}^M$. It is known that $D_1^\ast D_1$ and $D_1D_1^\ast$ have the same non-zero discrete spectrum. Thus, if zero is not an eigenvalue of $D_1D_1^\ast$ then we define the operator \[H_0^{(2)}= H_0-2g_1^\prime=-\frac{\mathrm{d}^2}{\mathrm{d}x^2}-V_0^{(2)}\] and deduce that it has eigenvalues $\{-\mu_k\}_{k=2}^M$. We call this the lifted operator, with lifted potential $V_0^{(2)}=V_0+2g_1^\prime$. It can be deduced that zero isn't an eigenvalue of $D_1D_1^\ast$ from the asymptotic behaviour of $g_1$
		\begin{equation} \label{eqn:gasmptot}
			g_1(x)=\begin{cases}
				-\sqrt{\mu_1},  &\text{ as } x\rightarrow \infty,\\
				\sqrt{\mu_1}, \ &\text{ as }x\rightarrow-\infty
			\end{cases}
		\end{equation}
		derived from solving for $\phi_1$ outside the support of $V_0$. We describe this in detail below.

		For the perturbed operator $H_V$ we proceed with the same argument. Consider its lowest eigenvalue $-\lambda_1$ and corresponding eigenfunction $\psi_1$ and define the operator $Q_1=D_1-f_1$, where $D_1$ is as before and $f_1=D_1\psi_1/\psi_1$. It follows that
		\begin{equation}
		\begin{split}
			&Q_1^\ast Q_1=H_V+\lambda_1,\\\label{EqnQ}
			&Q_1Q_1^\ast = H_V+\lambda_1-2f_1^\prime -2g_1^\prime
        \end{split}
		\end{equation}
		where we use that $f_1$ solves the Riccati-type equation
		\begin{equation}\label{eqn:Riccatif}
			\begin{split}
			f_1^\prime=\lambda_1-\mu_1 -V -f_1^2 -2g_1f_1.
			\end{split}
		\end{equation}
		
		From the behaviour of $\psi_1$ outside the support of both potentials, and using (\ref{eqn:gasmptot}), we have
		\begin{equation*}\label{eqn:asympf}
			f_1(x)=\begin{cases}
				-\sqrt{\lambda_1}+\sqrt{\mu_1}, &\text{ as } x\rightarrow \infty,,\\
				\sqrt{\lambda_1}-\sqrt{\mu_1}, &\text{ as } x\rightarrow -\infty.
			\end{cases}
		\end{equation*}
		Again, we use that $Q_1^\ast Q_1$ and $Q_1 Q_1^\ast$ have the same non-zero discrete spectrum and from \eqref{EqnQ} we observe that $Q_1^\ast Q_1$ has eigenvalues given by $\{-\lambda_k+\lambda_1\}_{k=1}^N$. We suppose that $0$ is an eigenvalue of $Q_1Q_1^\ast$. If so, then there would exist a nontrivial $\psi\in L^2(\R)$ with $Q_1Q_1^\ast \psi=0$ which would imply that $Q^\ast_1 \psi=0$ and thus $\psi$ solves
		\begin{align*}
			-\psi^\prime-g_1\psi-f_1\psi=0. 
		\end{align*}
		However, outside the supports of $V_0$ and $V$ this would result in the asymptotic behaviour $\psi(x)\sim \exp\left(\pm\sqrt{\lambda_k}x\right)$ as $x\rightarrow\pm \infty$, thus $\psi\notin L^2(\R)$ and $0$ is not an eigenvalue of $Q_1Q_1^\ast$. We conclude from (\ref{EqnQ}) that the new Schr\"odinger operator defined as
		\begin{align*}
			H_{V}^{(2)}= H_V -2f_1^\prime -2g_1^\prime=H_0^{(2)}-(V+2f_1^\prime)= H_0^{(2)}-V^{(2)} \text{ in }L^2(\R)
		\end{align*}
		has lifted discrete spectrum $\{-\lambda_k\}_{k=2}^{N}$.
		
		We repeat this process. We use the ground state of the operators $H_0^{(2)}$ and $H_V^{(2)}$ to factorise again and so on. Explicitly, we obtain a sequence of lifted operators $H_0^{(k)}$ and $H_V^{(k)}$ with ground states $u_k=(\prod_{j=1}^{k-1}D_j)\phi_{k}$ and $v_k=\prod_{j=1}^{k-1} (D_j-f_j)\psi_k$, respectively; where $g_j=u_j^\prime/u_j$, $D_j=\frac{\mathrm{d}}{\mathrm{d}x}-g_j$, $f_j=D_j v_j/v_j$, $Q_j=D_j-f_j$ and where $\phi_{k}$ and $\psi_k$ denote the eigenfunctions of $H_0$ and $H_V$, respectively. Analogous asymptotics and Riccati equations hold for these objects. We repeat this procedure until we exhaust the discrete spectrum of $H_0$.
		
		Following \cite{benguria1999simple}, we form a trivial integral inequality and use the Riccati equations, (\ref{eqn:Riccatif}), satisfied by $f_k$ and the asymptotic properties of $f_k$ and $g_k$ to obtain eigenvalue bounds. Consider for $k\leq M$
		\begin{equation*}\label{eqn:methodint}
		    \begin{split}
			0\leq \int_{\R} \left(V^{(k+1)}\right)^2\dd x=&\int_{\R} \left(V^{(k)}+2f_k^\prime\right)^2\dd x\\=&\int_{\R} \left(V^{(k)}\right)^2\dd x +4\int_{\R} f_k^\prime \left(V^{(k)}+f_k^\prime\right)\dd x,\\
			=&\int_{\R} \left(V^{(k)}\right)^2 \dd x+ 4\int_{\R} f_k^\prime \left(\lambda_k-\mu_k -f_k^2-2g_k f_k\right)\dd x,\\
			=&\int_{\R} \left(V^{(k)}\right)^2 \dd x+ 4 f_k (\lambda_k-\mu_k )\big\vert_{-\infty}^\infty - \frac{4}{3}f^3_k\big\vert_{-\infty}^\infty\\ &-4f_k^2g_k\big\vert_{-\infty}^\infty+4\int_{\R} g_k^\prime f_k^2\dd x
			\end{split}
		\end{equation*}
		from which the asymptotics give us explicit terms, we find
		\begin{equation*}\label{eqn:concluint}
		\frac{8}{3}\left(2\lambda_k^{\frac{3}{2}}-3\lambda_k\sqrt{\mu_k}+\mu_k^\frac{3}{2}\right)\leq\int_{\R} \left(V^{(k)}\right)^2 \dd x+ 4\int_{\R} g_k^\prime f_k^2 \dd x.
		\end{equation*}
		Using this inductively, we conclude that 
		\begin{equation*}
		    \sum_{k=1}^M\left(\lambda_k^{\frac{3}{2}}-\frac{3}{2}\lambda_k\sqrt{\mu_k}+\frac{\mu_k^\frac{3}{2}}{2}\right)\leq \frac{3}{16}\int_{\R} V^2 \dd x+ \frac{3}{4}\sum_{k=1}^N\int_{\R} g_k^\prime f_k^2 \dd x.
		\end{equation*}
		Finally, we note that the restrictions on $V_0$ and $V$ can be loosened. All we require from $V_0$ and $V$ is that the asymptotics used above hold true. In particular, we require that the ground states of the lifted operators $H_0^{(k)}$ and $H_V^{(k)}$, $u_k$ and $v_k$ given above, satisfy
		\begin{equation}
		    \begin{split}\label{asymptotics}
		    & u_k(x)\sim e^{\mp \sqrt{\mu_k}x}, \text{ and} \\
		    & v_k(x)\sim e^{\mp \sqrt{\lambda_k}x}
		    \end{split}
		\end{equation}
		as $x\rightarrow \pm \infty$, for all $k$. In this case there may be finite or infinitely many eigenvalues (accumulating at zero). 
		\begin{proposition}\label{prop:technical}
				Consider the Schr\"odinger operators $H_0$ and $H_V$ in $L^2(\R)$ with negative eigenvalues $\{-\mu_k\}_{k=1}^{M}$ and $\{-\lambda_k\}_{k=1}^N$ respectively such that the lifted eigenfunctions satisfy \eqref{asymptotics}. Then for any $K\leq M$
				\begin{align*}
				\sum_{k=1}^K\left[(\sqrt{\lambda_k}-\sqrt{\mu_k})^3+\frac{3}{2} \sqrt{\mu_k} \left(\sqrt{\lambda_k}-\sqrt{\mu_k}\right)^2\right] \leq& \frac{3}{16}\int_{\R} V(x)^2 \dd x+\mathcal{E}_K
				\end{align*}
				where $\mathcal{E}_K=\frac{3}{4}\sum_{k=1}^K\int_{\R} g_k^\prime f_k^2\dd x$ and 
				with $g_k$ and $f_k$ as above. 
		\end{proposition}
		For the P\"oschl-Teller potential we shall see that after factorising $M$ times, the potential $V_0^{(M+1)}$ is strictly negative. In this case, one can remove it with a variational argument and continue on with the scheme of Benguria-Loss. Though this does not always occur, the remark below explains how we can formally apply the result of Frank-Simon-Weidl.
		\begin{remark}
		After we have exhausted all the $M$ eigenvalues of $H_0$ we have the operator \[H_V^{(M+1)}=-\frac{\mathrm{d}^2}{\mathrm{d}x^2}-V_0^{(M+1)}-V^{(M+1)}.\] Since $H_0^{(M+1)}$ has no negative eigenvalues, we apply the bound of Frank-Simon-Weidl \eqref{eqn:franksimonweidl} to deduce the existence of $\beta>0$ such that
		\begin{equation}\label{eqn:remarkFSW}
		    \beta^{-2}\sum_{k=M+1}|E_k(H_V)|^{3/2}\leq \frac{3}{16}\int_{\R}\left(V^{(M+1)}\right)^2\dd x.
		\end{equation}
		Thus the additional eigenvalues can be included into the bound in the proposition above by expanding the right-hand side in \eqref{eqn:remarkFSW}.
		\end{remark}
		The majority of the above analysis holds for the Robin operators. The following remark makes clear the similarities and distinctions.
		\begin{remark}\label{rem:robinrem}
		    The operators $H_{0,\sigma}$ and $H_{V,\sigma}$ in $L^2(\R_{+})$ with boundary conditions $\phi^\prime(0)-\sigma\phi(0)=0$ factorise similarly. In particular, (\ref{EqnD}) and (\ref{EqnQ}) hold with $g_1$ and $f_1$ analogously defined according to the ground states. For compactly supported potentials the same asymptotic behaviour holds as $r\rightarrow\infty$ with $g_1(0)=\sigma$ and $f_1(0)=0$ at the boundary. Thus, in particular, we find 
		    \begin{equation} \label{ineqrobin}
					\left(\sqrt{|E_1(H_{V,\sigma})|}-\sqrt{|E_1(H_{0,\sigma})|}\right)^3 \leq \frac{3}{8}\int_{\R_{+}} V^2\dd r+\frac{3}{2}\int_{\R_{+}} g_1^\prime f_1^2\dd r
			\end{equation} 
			which holds for more general $V_0, V$ as in the above proposition (satisfying condition \eqref{asymptotics} as $r\rightarrow +\infty$). As noted in \cite{exner2014some}, an important distinction occurs after the first factorisation. The lifted Schr\"odinger operators $H^{(2)}_{0,\infty}$ and $H^{(2)}_{V,\infty}$ have Dirichlet conditions at $0$.
		\end{remark}
		The remainder of this paper is largely concerned with removing the term $\mathcal{E}_K$ in the above proposition. A sufficient condition is for $g_k^\prime$ to be negative almost everywhere; that is, the ground states of the operators $H_0^{(k)}$ are log-concave. This is equivalent to $H_0^{(k+1)}\geq H_0^{(k)}$ in the form sense for all $k\leq K$. An ordering of this sort seems like a reasonable assumption, given that the energies are lifted at each step of factorisation. In the next two sections, we are be able to show this for $K=1$ when $V_0$ is single well and for all $K$ for specific examples. However, this stronger property does not hold in general, see Remark \ref{rmk:doublewell}. 
		\begin{remark}\label{rmk:doublewell}
	        Consider the case where $V_0$ is a double-well potential, it can be seen that the corresponding ground state is also double-well. Double-well functions are not quasi-concave and therefore are not log-concave. An analytic proof of this can be found for the hyperbolic double-well potential in \cite{downing2013solution}.
	    \end{remark}
	\section{Proof of theorems \ref{thm:firsteigen} and \ref{thm:Robincase}} \label{section:thm12}
	    Log-concavity of the ground state has been explored by many authors. On a bounded domain, for the Dirichlet and Neumann Laplacians this is a simple fact (for the Robin case there are exceptions, see \cite{andrews2020non}). For Schr\"odinger operators on a bounded domain, H.Brescamp and E.Lieb showed as a consequence of a functional inequality in \cite{brascamp2002extensions} that this holds for concave potentials. An application, and simpler proof, of this fact was presented in \cite{singer1985estimate} where the authors used it to derive a lower bound on the difference of the first two eigenvalues, the \enquote*{fundamental gap}. The single-well condition arose with the same application in the paper of M.Ashbaugh and R.Benguria \cite{ashbaugh1989optimal}, though log-concavity was not explicitly shown or used.
	    
	    Works by B.Baumgartner, H.Grosse and A.Martin have shown that log-concavity holds for the ground states of operators on $L^2(\R)$ of the form $H_V$ where $V$ is \enquote*{$V_0-$concave}, see \cite{baumgartner1990relative}. Originally this was shown for Hardy-type $V_0$ in \cite{baumgartner1986level}, a proof which was later simplified in \cite{ashbaugh1988log}. For our application we only need to look on $H_0$. 
	
		To prove Theorem \ref{thm:firsteigen} we apply Proposition \ref{prop:technical} whilst removing the last term in the bound. Log-concavity is sufficient for the latter point, but the concavity condition on $V_0$ of Brescamp-Lieb causes problems with the asymptotics required \eqref{asymptotics}. Thus, we look at $V_0$ that are single well and prove log-concavity in this case. The remainder of the proof relies on a variational argument.
		\begin{lemma} \label{lemsingle}
		    Suppose that $(1+|x|)V_0\in L^1(\R)$ and that $V_0$ is positive and single-well, then the eigenfunction $\phi$ corresponding to the lowest eigenvalue $-\mu$ of $H_0$ on $L^2(\R)$ is log-concave.
		\end{lemma}
		\begin{proof}
		    Under the condition that $(1+|x|)V_0\in L^1(\R)$, by considering Jost functions, it follows that $\varphi$ has exponential decay obeying \eqref{asymptotics} (see \cite{Chadan1989}). With this in hand, we define the log-derivative of the ground state as $g=\phi^\prime/ \phi$. It is weakly differentiable (hence continuous), obeys the Riccati equation \eqref{EqnRiccatig} and tends to $\mp \sqrt{\mu}$ as $x\rightarrow\pm\infty$. 
		    
		    Suppose that $g^\prime>0$ (almost everywhere) on some open interval $I_1$. If there is no disjoint open interval where $g^\prime<0$ beforehand then the asymptotics tell us that $g>\sqrt{\mu}$ on $I_1$ and 
			\begin{equation*}
		        V_0(x)=\mu-g^2(x)-g^\prime(x)-g(x)<-g^\prime(x)<0 \text{ on $I_1$}
			\end{equation*}
			 using the Riccati equation, which contradicts the positivity of $V_0$. Hence, there exists a open interval $I_0$ on which $g^\prime<0$. Furthermore, we can assume that $g>-\sqrt{\mu}$ on $I_1$, otherwise, using the asymptotics, one can find that $V_0<0$ as before and obtain a contradiction. Therefore, using the behaviour of $g$ as $x\rightarrow \infty$ there must exist another disjoint open interval $I_2$ after, on which $g^\prime<0$. 
			 
			 We deduce the existence of three points $x_0<x_1<x_2$ with $x_k\in I_k$, $k=1,2,3$  for which  $g(x_1)=g(x_2)=g(x_2)\eqcolon c$. Using the Riccati equation and the respective sign of $g^\prime$ we find
			\begin{equation}\label{eqn1}
			    \begin{split}
			        V_0(x_1)>\mu-c^2, \\
			        V_0(x_2)<\mu-c^2, \\
			        V_0(x_3)>\mu-c^2. 
			    \end{split}
			\end{equation}
			If $V_0$ is continuous we are finished as this contradicts the single-well property. Otherwise, by continuity of $g$, the inequalities \eqref{eqn1} extend a.e. to neighbourhoods about $x_1$, $x_2$ and $x_3$ respectively. We conclude that $g^\prime\leq 0$. 
		\end{proof}
		Assume that $V_0$ satisfies the conditions of the lemma and that $V\in C_0^\infty(\R)$ and $V\geq 0$. We apply Proposition \ref{prop:technical} for a bound on the first eigenvalues ($K=1$), removing the last term owing to the negativity of $g_1^\prime$. To obtain the full statement of Theorem \ref{thm:firsteigen} we use a variational argument.
		
		As a consequence of Lemma \ref{lemsingle}, the lifted operator $H^{(2)}_V$ satisfies
		\begin{equation*}
		    H_V^{(2)}\geq -\frac{\mathrm{d}^2}{\mathrm{d}x^2}-V_0-V^{(2)}=H_{V^{(2)}}
		\end{equation*}
		in the quadratic form sense. We look on the operator $H_{V^{(2)}}$ and denote it by $H^{(2)}$, thus we have $|E_2(H_V)|=|E_1(H^{(2)}_V)|\leq |E_1(H^{(2)})|$. We lift $H^{(2)}=H_{V^{(2)}}$ again, according to Section \ref{Section:facscheme}, dispose of $g_1$ and define the lifted operator
		\begin{equation*}
		    H^{(3)}=H_{V^{(2)}}^{(2)}\geq -\frac{\mathrm{d}^2}{\mathrm{d}x^2}-V_0-(V^{(2)})^{(2)}
		    =H_{(V^{(2)})^{(2)}}.
		\end{equation*}
		Repeating this, we obtain a sequence of operators $H^{(k)}=H_{V_k}$ with $V_k=((V^{(2)})^{(2)\cdots})^{(2)}$, iterated $k-1$ times, and such that $|E_k(H_V)|\leq |E_1(H^{(k)})|$. Once $|E_k(H_V)|\leq |E_1(H_0)|$ then we stop. Applying the inductive integral method for this sequence of potentials, as in Proposition \ref{prop:technical}, we form
		\begin{equation}\label{eqn:proofthm1}
		    \begin{split}
		    \sum_k\left[\sqrt{|E_k(H_V)|}-\sqrt{|E_1(H_0)|} \right]_{+}^3&\leq \sum_k \left(\sqrt{|E_1(H^{(k)})|}-\sqrt{|E_1(H_0)|}\right)^3
		    \\&\leq \frac{3}{16}\int_{\R} V(x)_+^2 \dd x.
		    \end{split}
		\end{equation}
		By a standard variational argument we can extend the bound to a more general class of $V$ where only the positive part of $V$ appears on the right-hand side. This concludes the proof for $\gamma=3/2$. To extend the bound to $\gamma\geq 3/2$ we use the argument of M.Aizenman and E.Lieb, which relies on the identity
		\begin{equation}\label{eqn:AizenmanIdentity}
		    \lambda_{-}^\gamma=C_{\gamma,\delta}\int_0^\infty \kappa^{\gamma-\delta-1}\left(\lambda+\kappa\right)^\delta_{-}\dd\kappa
		\end{equation}
		which holds for $\gamma\geq \delta$, where $C_{\gamma,\delta}$ is some positive constant related to the beta function, see \cite{aizenman1978semi}. Applying the identity together with the fact that $\sqrt{a+b}\leq \sqrt{a}+\sqrt{b}$ for $a,b\geq 0$, we find that for any $\gamma\geq 3/2$
		\begin{align*}
		   \left[\sqrt{|E_k(H_V)|}-\sqrt{|E_1(H_0)|} \right]_{+}^{2\gamma}&=C_{2\gamma,3}\int_0^\infty \kappa^{2\gamma-4}\left(\kappa+\sqrt{|E_1(H_0)|}-\sqrt{|E_k(H_V)|}\right)^3_{-}\dd \kappa\\
		    &\leq C_{2\gamma,3}\int_0^\infty \kappa^{2\gamma-4}\left(\sqrt{|E_1(H_0)|+\kappa^2}-\sqrt{|E_k(H_V)|}\right)^3_{-}\dd \kappa
		 \end{align*}
		 after summing over $k$ and using \eqref{eqn:proofthm1} we find that 
		 \begin{align*}
		    \sum_{k} \left[\sqrt{|E_k(H_V)|}-\sqrt{|E_1(H_0)|} \right]_{+}^{2\gamma}&\leq \frac{3C_{2\gamma,3}}{16}\int_0^\infty \kappa^{2\gamma-4}\int_{-\infty}^\infty (V(x)-\kappa^2)_{+}^2\dd x\dd \kappa\\
		    &=\frac{3C_{2\gamma,3}}{32}\int_{-\infty}^\infty\int_0^\infty \kappa^{\gamma-\frac{1}{2}-2}(-V(x)+\kappa)_{-}^2\dd \kappa\dd x\\
		    &=\frac{3C_{2\gamma,3}}{32C_{\gamma+1/2,2}}\int_{-\infty}^\infty V(x)_{+}^{\gamma+\frac{1}{2}}\dd x
		\end{align*}
		where we have used a change of variables in $\kappa$ and applied the identity \eqref{eqn:AizenmanIdentity} once more. This concludes the proof of Theorem \ref{thm:firsteigen}. The following remark explains how the result follows for the Robin operator on the half-line. 
		\begin{remark}
		On the half-line for the operator $H_{0,\sigma}$ with $\sigma\in \R$, an analogous version of Lemma \ref{lemsingle} holds for non-increasing potentials by near-identical proof. In this case the asymptotic requirement for \eqref{ineqrobin} at infinity is satisfied when $V_0\in L^1(\R_{+})$, which follows by work in \cite{demirel2011trace}. From Remark \ref{rem:robinrem}, the bound in Theorem \ref{thm:Robincase} for $\gamma=3/2$ follows by the same variational idea as above, where at each step we also use that $H_{V,\sigma}\leq H_{V,\infty}$. The lifting to $\gamma\geq 3/2$ follows an identical argument.
		\end{remark}
	\section{Proof of theorems \ref{thm:soliton} and \ref{thm:Coulomb}}\label{section:polcou}
		Finally, we apply our result from Section 2 to two well-known examples of $H_0$. For the P\"oschl-Teller potential, this will be a direct result of the explicit factorisation and application of Proposition \ref{prop:technical}. For the Coulomb potential we will need to be careful about the Dirichlet boundary conditions. In both cases we will see that we can obtain superior estimates to those in theorems \ref{thm:firsteigen} and \ref{thm:Robincase}.
		\subsection{P\"oschl-Teller Potential}
			Consider the case of the potential
			\begin{equation*}
				V_0(x)=\nu(\nu+1)\sech^2(x)
			\end{equation*}
			with $\nu>0$. It has an explicit factorisation according to the first order operator $D=\frac{\mathrm{d}}{\mathrm{d}x}+\nu \tanh(x)$. Following the factorisation method (outlined in Section \ref{Section:facscheme}), its lifted potential is given by
			\begin{equation*}
				V_0^{(2)}(x) = \nu (\nu-1)\sech^2(x),
			\end{equation*}
			thus we can see that $H_0^{(2)}$ has the same shaped potential as $H_0$, with a change of parameter. It follows that 
			\begin{equation*}
				V_0^{(k)}(x) = (\nu-k+1)(\nu-k+2)\sech^2(x).
			\end{equation*}
			Using this scheme we can explicitly compute the negative eigenvalues to be $E_k(H_0)=-(\nu-k+1)^2$, $k=1,\ldots,\lceil \nu \rceil$ as well as the corresponding eigenfunctions, see \cite{frankLaptevbook}. Each lifted operator has  ground state given by $u_k=\cosh^{-\nu+k}(x)$ with log-derivative given by $g_k=-(\nu-k+1)\tanh(x)$. Thus $V_0$ satisfies the requirements of Proposition \ref{prop:technical} and at each stage its lifted ground state is log-concave. If $V\in C_0^\infty(\R)$ then $V$ also satisfies the conditions \eqref{asymptotics}. Thus, we conclude that
			\begin{equation*}
				\sum_{k=1}^{\lceil\nu\rceil}\left(\sqrt{|E_k(H_V)|}-(v-k+1)\right)^3+\sum_{k=\lceil\nu\rceil+1}|E_k(H_V)|^\frac{3}{2} \leq \frac{3}{16}\int V(x)^2_{+}\dd x.
			\end{equation*}
			Where the standard approximation and variational arguments extend this to more general $V$. The additional terms on the left-hand side comes from removing $V_0^{(\lceil \nu\rceil)}\leq 0$ after the $\lceil\nu\rceil$ step and following Remark \ref{eqn:remarkFSW}.  
			
			The second bound \eqref{eqn:polsch2nd} is obtained from the full statement of the proposition by removing the difference terms and applying the standard lifting argument of Aizenman-Lieb \cite{aizenman1978semi} to $\gamma\geq 3/2$. 
	\subsection{Coulomb potential}
		Now consider the case of the operator with Coulomb potential on the positive real line with Dirichlet boundary condition at $0$. That is, 
		\begin{equation*}
			H_{0,\infty}=-\frac{\mathrm{d}^2}{\mathrm{d}r^2}+\frac{\nu(\nu+1)}{r^2}-\frac{\kappa}{r} \text{ in } L^2(\R_{+}), 
		\end{equation*}
		where $\nu> -1/2$ and $\kappa>0$. The explicit factorisation is given according to the first order operator $D=\frac{\mathrm{d}}{\mathrm{d}r}-\frac{(v+1)}{r}+\frac{\kappa}{2(v+1)}$, from which the lifted potential is found to be 
		\begin{equation*}
		        V^{(2)}_0(r)=-\frac{(\nu+1)(\nu+2)}{r^2}+\frac{\kappa}{r}.
		\end{equation*}
		We see that $H_{0,\infty}^{(2)}$ has the same shaped potential with different parameters. Iterating this, we find that 
		\begin{equation*}
		        V^{(k)}_0(r)=-\frac{(\nu+k-1)(\nu+k)}{r^2}+\frac{\kappa}{r}.
		\end{equation*}
		This scheme can be used to explicitly calculate the negative eigenvalues as $E_k(H_{0,\infty})=-\kappa^2/4(\nu+k)^2$, $k\in \mathbb{N}$ and the associated eigenfunctions, see \cite{frankLaptevbook}. The lifted ground states are given by $u_k(r)=r^{\nu+1}e^{-r\sqrt{|E_k(H_{0,\infty})|}}$, these are log-concave with $g_k^\prime =-(\nu+k)/r^2$. Though these $u_k$ do not obey the asymptotic requirements of Remark \ref{rem:robinrem} we still have that $g_k\rightarrow -\sqrt{|E_k(H_{0,\infty})|}$ at infinity. However, we cannot directly apply the ideas of the Robin case since the value of $g_k$ at $0$ is potentially problematic. We treat this with care, following \cite{laptev2021factorisation}.
		
	     We revert to the scheme in Section \ref{Section:facscheme} and denote by  $\lambda_k=-E_k(H_{V,\infty})$. Assume that $V\in C_0^\infty(\R_{+})$ and consider the behaviour of the ground state of $H_{V,\infty}$, $\psi_1$, outside of the support of $V$. Then $\psi_1$ satisfies
		\begin{equation*}
			-\psi_1^{\prime\prime}(r)+\left(\frac{v(v+1)}{r^2}-\frac{\kappa}{r}\right) \psi_1(r)=-\lambda_1 \psi_1(r).
		\end{equation*}
		The solution can be found in terms of the Whittaker functions. Using the asymptotic expansion of these it follows that $f_1$ (defined as in Section \ref{Section:facscheme}) satisfies 
		\begin{equation*}
			f_1(r)=\begin{cases}
				\frac{r \left(-\kappa+2 \sqrt{\lambda_1} \nu+2 \sqrt{\lambda_1}\right) \left(\kappa+2 \sqrt{\lambda_1} \nu+2 \sqrt{\lambda_1}\right)}{4 (\nu+1)^2 (2 \nu+3)}+\mathcal{O}\left(r^2\right), \ &\text{ as }r\rightarrow 0\\
				-\sqrt{\lambda_1}+\frac{\kappa}{2 r\sqrt{\lambda_1} }+\mathcal{O}\left(r^{-2}\right), &\text{ as } r\rightarrow \infty.
			\end{cases}
		\end{equation*}
		We use this in the inductive method of Section \ref{Section:facscheme} with $g_1^\prime<0$. Together with the fact that $H_{0,\infty}$ lifts to another Coulomb operator we can check asymptotic requirements are met for all lifted potentials. Thus the following inequality holds,
		\begin{equation*}
			\sum_{k=1}\left(\sqrt{|E_k(H_V)|}-\frac{\kappa}{2(\nu+k)}\right)^3\leq \frac{3}{8}\int_{\R_{+}} V(r)^2 \dd r
		\end{equation*}
		which extends to a larger class of $V$ as before.
	\begin{remark}
		Both examples in this section are known as shape-invariant potentials; that is, at each step of factorisation the new potentials have the same shape but differ by some parameter. Such potentials are fully factorisable, which in these cases allows us to apply the full version of Proposition \ref{prop:technical}. See \cite{gangopadhyaya2008generating} for further examples of shape-invariant potentials. 
 	\end{remark}

	\subsection*{Acknowledgments} The author would like to thank Lukas Schimmer for useful discussions and Ari Laptev for his supervision and insight.
		
		\bibliographystyle{alpha-mod}
		\bibliography{pertbib.bib}
	\end{document}